\documentclass[11pt]{article}
\usepackage{amsmath}
\usepackage{xcolor}
\usepackage{amssymb}
\usepackage{amsthm}
\usepackage{cite, comment}
\usepackage{pgf,tikz,pgfplots}
\pgfplotsset{compat=1.14}
\usepackage{mathrsfs}
\usetikzlibrary{arrows}
\usepackage{mathrsfs}
\usepackage{enumitem}

\newtheorem{theorem}{Theorem}
\newtheorem{thmx}{Theorem}

\newtheorem{lemma}[theorem]{Lemma}
\newtheorem{proposition}[theorem]{Proposition}

\theoremstyle{definition}

\newtheorem{question}[theorem]{Question}

\newtheorem*{claim*}{Claim}

\newenvironment{proofofclaim}[1]{\par\noindent\underline{Proof of claim:}\space#1}{\hfill $\blacktriangle$}

\newcommand{\R}{\mathbb{R}}

\newcommand{\N}{\mathbb{N}}

\newcommand{\F}{\mathcal{F}}
\newcommand{\G}{\mathcal{G}}
\renewcommand{\H}{\mathcal{H}}

\newcommand{\cl}[1]{\overline{#1}}
\renewcommand{\int}{\text{int\,}}
\newcommand{\bd}{\partial}
\newcommand{\diam}{\text{diam\,}}

\newcommand{\mesh}{\text{mesh\,}}
\newcommand{\T}{\mathsf{T}}
\newcommand{\CT}{\mathsf{CT}}
\newcommand{\C}{\mathsf{C}}
\newcommand{\LEO}{\mathsf{LEO}}
\newcommand{\DP}{\mathsf{DP}}
\newcommand{\LC}{\mathsf{LC}}
\newcommand{\M}{\mathsf{M}}
\newcommand{\SH}{\mathsf{SH}}

\def\subjclass#1{\par\medskip
\noindent\textbf{Mathematics Subject Classification (2020):} #1}
\def\keywords#1{\par\medskip
\noindent\textbf{Keywords.} #1}

\begin{document}
\makeatletter

\title{Chaos on Peano continua}
\author{Klára Karasová
\footnote{Email: karasova@karlin.mff.cuni.cz; Orcid: 0000-0002-9019-443X.
This paper was supported by the grant GA UK n. 129024}
\ and
Benjamin Vejnar
\footnote{Email: vejnar@karlin.mff.cuni.cz; Orcid: 0000-0002-2833-5385. 
This paper was supported by the grant GACR 24-10705S}
\\ \\
Faculty of Mathematics and Physics \\
Charles University \\
Prague, Czechia\\ \\
}

\maketitle

\begin{abstract}
Generalizing the result of Agronsky and Ceder (1991), we prove that every Peano continuum admits a continuous transformation that is exact Devaney chaotic; that is, it has a dense set of periodic points, and every nonempty open set covers the entire space in finitely many iterations. 
We identify a natural class of Peano continua, containing all one-dimensional continua and all absolute neighborhood retracts, which allows us to create locally small perturbations. 
Using this method, we prove that within these specific classes of continua, exact Devaney chaotic systems are dense in all chain transitive systems, mixing systems are generic among chain transitive systems and shadowing is generic among all continuous systems.
\end{abstract}


\section{Introduction} 
% Základní myšlenka
In this paper, we study dynamical systems $(X, T)$, where $X$ is a compact metric space and $T:X\to X$ is a continuous transformation.
For many spaces $X$, no nontrivial transformations exist -- for instance, on Cook continua, where the only continuous transformations are the identity and constant maps.
A lot of attention has been given to the existence of some \emph{nontrivial} or \emph{chaotic} dynamics on \emph{specific} compact metric spaces. 
In the past, the focus was given to the interval, the Cantor set and manifolds. In recent decades, dendrites, Peano continua or the Lelek fan have also become popular.
Moreover, in many cases, it turns out that complicated behavior is typical in the sense of Baire category. Let us briefly mention some of the related results.

% Konkrétní výsledky
% interval
The interval dynamics is the most studied with plenty of results.
The typical dynamical behavior on $I=[0,1]$ was investigated e.g. in \cite{AgronskyBrucknerLaczkovich} where it is shown that the union of all scrambled sets of a typical self-map on $I$ is first category and of measure zero. Also, the closure of the union of all attracting sets of a typical map is first category and measure zero, while for a transitive map it is necessarily the full $I$, and thus a typical map on $I$ is far from being transitive. 
Recently, it was proved that transitive transformations of $I$ form a set homeomorphic to the Hilbert space \cite{HeLiYang}.

% Cantor
The Cantor set  $C$ serves as the best space for demonstrating chaotic dynamics, see e.g. \cite{BoronskiKupkaOprocha} for a mixing, completely scrambled system on $C$. A typical homeomorphism and a typical map of $C$ were completely described up to conjugacy \cite{KechrisRosendal, BernardesDarji} (see also \cite{AkinGlasnerWeiss, KupkaOprocha} for a more familiar description). 
The typical homeomorphism on $C$ has zero entropy \cite{glasnerweiss}. Further, a typical self-map of $C$ has no periodic points and thus it is not Devaney chaotic on any subsystem, even though there is a dense subset of the space of self-maps of $C$ formed by maps with infinite entropy that are Devaney chaotic on some subsystem \cite{DanielloDarji}.

% manifolds:
One of the early related results dealing with manifolds is by Besicovitch, who constructed a transitive homeomorphism of the 2-dimensional sphere. Oxtoby found out that transitivity is in fact typical when restricted to continuous measure-preserving homeomorphisms.
%Besicovitch: %A problem on topological transformation of the plane, Fund. Math. 28(1937), 61-65
Methods of Oxtoby were later reused for proving the existence of chaotic homeomorphisms on $n$-dimensional manifolds ($n\geq 2$) or for Menger manifolds  \cite{KatoKawamuraTuncaliTymchatyn, AartsDaalderop}.
See the book \cite{AlpernPrasad} for more details on this topic.
%Leaving the case of specific spaces, interesting results were obtained for manifolds of dimension at least two:
Connected manifolds admit everywhere chaotic homeomorphisms \cite{Kato}, where being everywhere chaotic is a weaker condition than being topologically mixing.
Last, but not least, shadowing is generic in homeomorphisms of a manifold \cite{PilyuginPlamenevskaya}.

% Lelek fan
The question of (non)existence of chaotic maps seems to be more challenging for spaces that are arc-wise connected (unlike the Cantor set, for example) and yet are not locally connected (unlike Peano continua and manifolds). Nevertheless, there is a construction of a transitive homeomorphism on the Lelek fan \cite{BanicErcegKennedy}, which is an example of an arc-wise connected but not locally connected continuum. The Lelek fan even admits a completely scrambled weakly mixing homeomorphism \cite{Oprocha}.

% dendrites
In \cite{BoronskiMincStimac} it was shown that dendrites admit topologically mixing self-maps using a factorization through $[0,1]$.
In the invertible case, the situation is even more restrictive, since only very specific dendrites admit a sufficiently rich homeomorphism group. Every universal Wa\.zewski dendrite $D_n$, $n\geq 3$, admits a homeomorphism with the shadowing property \cite{CarvalhoDarji}.

% Peano one-dim.
Recently it was proved that shadowing is generic among all self-maps in the case of dendrites \cite{BrianMeddaughRaines} and even more generally in the case of so called graphites \cite{Meddaugh}. It was asked in the latter paper whether the class of graphites coincides with the class of all one-dimensional Peano continua. This is indeed the case, as 
every Peano continuum of dimension one admits retractions which are arbitrarily close to identity and whose range is a graph (see \cite{Manka} or \cite[Theorem 4.3]{KrupskiOMiljanowskiUngeheuer}).
Thus by \cite{Meddaugh} shadowing is a generic property in the space of continuous self-maps of any one-dimensional Peano continuum.
The generic map of a one-dimensional Peano continuum also has the set of chain-recurrent points zero-dimensional \cite{KrupskiOMiljanowskiUngeheuer}.

% konečně dimenzionální Peano continua
In 1991, Agronsky and Ceder constructed for every finitely dimensional Peano continuum $X$ an exact self-map of $X$ \cite{AgronskyCeder} (i.e. a map such that every nonempty open set covers the space in a finite number of iterations). However, articles dealing with the existence of nontrivial dynamics on Peano continua are rare.
% Naše výsledky
In this paper, we further develop the construction of Agronsky and Ceder in two directions. First, we are able to capture all Peano continua (even the infinitely dimensional) and second, we get a dense set of periodic points of the resulting exact map. Thus, we get:

\begin{thmx}[with details in Theorem \ref{existenceofexactDevaney}]\label{Main}
Every Peano continuum admits an exact Devaney chaotic map.
\end{thmx}

See the paper \cite{KwietniakMisiurewicz} for the motivation behind the notion of exact Devaney chaos.
Theorem \ref{Main} can be viewed as a parallel result to a classical one by 
Rolewicz, who found a bounded linear operator on a separable infinite-dimensional Hilbert space \cite{Rolewicz} and he asked whether every infinite-dimensional separable Banach space admits a transitive operator. Consequently,
a positive answer was given independently by several authors in between 1997--1999 \cite{Ansari, Bernal-Gonzalez, BonetPeris}.

We denote by \(\C(X,Y)\) the space of all continuous self-maps \(f : X \to Y\), equipped with the topology of uniform convergence. Let us denote by \(\LC(X,Y)\) the collection of all \(f \in \C(X,Y)\) that are locally constant on a dense (open) subset of \(X\); that is, \(f \in \LC(X,Y)\) if and only if there exists a dense set \(D \subseteq X\) such that for every \(x \in D\) there is a neighborhood \(U\) of \(x\) such that \(f|_U\) is constant.  If \(X = Y\), we simply write \(\LC(X)\) instead of \(\LC(X,Y)\).

Clearly, \(\LC(X,Y)\) contains all constant functions, but it also includes, for example, the Devil’s staircase function (also known as the Cantor ternary function) when \(X = Y = [0,1]\). Having this in mind, it can be easily seen that $\LC([0,1])$ is, in fact, dense in $\C([0,1])$. That is no exception, as we show in Theorem \ref{LCpodminka} that $\LC(X)$ is dense in $\C(X)$ whenever $X$ is an absolute neighborhood retract (in particular, if $X$ is a compact manifold), or when $X$ is a one-dimensional Peano continuum.

Nevertheless, despite Theorem \ref{LCpodminka}, we do not know whether $\LC(X)$ is dense in $\C(X)$ for every Peano continuum $X$, and we suspect that it may not be. This motivates the following question, whose positive answer would have direct consequences -- for example, for \cite[Question 10.2(1)]{KrupskiOMiljanowskiUngeheuer}.

\begin{question}\label{otazkaPeano}
Does there exist a nondegenerate Peano continuum $X$ for which one of the following conditions holds?
\begin{itemize}[noitemsep]
\item $\LC(X)$ is not dense in $\C(X)$.
\item $id_X\notin\cl{\LC(X)}$.
\item $id_X$ is an isolated point of $\C(X)$.
\end{itemize}
\end{question}

The reason why we introduce the class $\LC(X)$ is that the maps in $\LC(X)$ are very flexible and allow us to produce locally small perturbations. Consequently, if $\LC(X)$ is dense in $\C(X)$, then any $f\in \C(X)$ can be uniformly approximated by flexible maps and thus we are able to use a generalized version of the window perturbation method (see e.g. \cite{BobokCincOprochaTrubetzkoy}). 

However, such an abundance of flexible maps has structural consequences. In Proposition \ref{CTnowheredense}, we show that if $\LC(X)$ is dense in $\C(X)$, then $\CT(X)$ -- the subspace of all chain transitive self-maps on X -- is nowhere dense in $\C(X)$. Thus, transitivity can never be dense, let alone generic, in $\C(X)$, if $\LC(X)$ is dense in $\C(X)$. Note how this and Theorem \ref{LCpodminka} align with the known results listed above.

Given these limitations, we restrict our attention to $\CT(X)$, studying typical maps within this subspace, in contrast to prior work, which focus on how far a typical map in $\C(X)$ is from being transitive. And indeed, we are able to establish the following immediate consequence of Theorems  8, 11, 13, and 16:

\begin{thmx}\label{thmmain3}
Let $X$ be a one-dimensional Peano continuum, or an absolute neighborhood retract, or more generally, let $X$ be a Peano continuum with $\LC(X)$ dense in $\C(X)$. Then
\begin{itemize}[noitemsep]
\item exact Devaney chaotic maps are dense in $\CT(X)$,  
\item generic map $f\in \CT(X)$ is mixing,
\item generic map $f\in \C(X)$ has shadowing.
\end{itemize}
\end{thmx}

Since every compact manifold is an absolute neighborhood retract, Theorem \ref{thmmain3} applies directly in this setting.
Some aspects of Theorem \ref{thmmain3} with $X=[0,1]$ are related to \cite{BobokCincOprochaTrubetzkoy}.
Other relations regarding shadowing on some class of 1-dimensional continua (graphs, dendrites, arc-like continua) are in \cite{OprochaEtAlShadowing}.

\section{Preliminaries}
For a topological space $X$ and $A\subseteq X$, we denote by $\int A$, $\cl{A}$, and $\bd A$ the interior, the closure, and the boundary of $A$, respectively. In a metric space, we denote by $\diam A$ the diameter of $A$. If $\F$ is a family of subsets of $X$, we denote $\mesh \F = \sup\{\diam F;\,F\in \F\}$. Further, we denote by $B(x,\varepsilon)$ the open ball centered at $x\in X$ with radius $\varepsilon>0$.

A continuum is a compact connected metrizable topological space. 
An absolute retract is any space homeomorphic to a retract of the Hilbert cube.
If $A\subseteq X$, we say that $A$ is a neighborhood retract of $X$ if there is a neighborhood $U$ of $A$ in $X$ and a retraction $r:U\to A$. An absolute neighborhood retract is any compact space homeomorphic to a neighborhood retract of the Hilbert cube.
We use the notion of topological dimension according to \cite{EngelkingDim}.

\subsection{Peano continua}
A Peano continuum is a continuum which is locally connected. 
Throughout the whole article, we heavily use that every Peano continuum admits a compatible convex metric \cite[Theorem 8]{Bing}. Recall that a metric $d$ on a space $X$ is called \emph{convex} if for every $x,y\in X$ there exists $A\subseteq X$ isometric to $[0,d(x,y)]$ such that $x,y\in A$. Note that a space with a convex metric has all balls (no matter if open or closed) connected.

We say that a metric space $X$ has \emph{the property S} if for every $\varepsilon>0$ there is a finite cover $\F$ of the space $X$ consisting of connected (but not necessarily open) sets satisfying $\mesh\F<\varepsilon$. It is an easy consequence of compactness that every Peano continuum has the property S. Further, if $X$ is a metric space and $\varepsilon>0$, we say that $X$ is \emph{$\varepsilon$-partitionable} if there exists \emph{an $\varepsilon$-partitioning} $\G$ of the space $X$, i.e. there exists $\G$ -- a finite family of pairwise disjoint open connected subsets of $X$  satisfying $\mesh \G<\varepsilon$ and $\cl{\bigcup \G} =X$. We say that $X$ is \emph{partitionable}  if it is $\varepsilon$-partitionable for every $\varepsilon>0$.

The following theorem will be used later. Although we were unable to locate it in the existing literature, there are related results available (see, for example, \cite{MayerOversteegenTymchatyn}).

\begin{proposition}\label{partitions}
Let $X$ be a Peano continuum and $\varepsilon>0$. Then there exists a finite cover $\F$ of $X$ by Peano continua such that $\mesh\F\leq\varepsilon$ and the sets in $\F$ have nonempty pairwise disjoint interiors.
\end{proposition}

\begin{proof}
Since $X$ has the property S, it is partitionable by \cite[Theorem 1]{Bing}. Since $X$ is partitionable, there is $\G$ an $\varepsilon$-partitioning of $X$ such that each member of $\G$ has the property S by \cite[Theorem 4]{Bing}.  Let $\F:=\{\cl{G},G\in\G\}$. Clearly $\F$ is a finite cover of $X$ and $\mesh\F<\varepsilon$. Moreover, each member of $\F$ is easily seen to be a continuum with the property S and thus is a Peano continuum by \cite[Theorem 8.4]{nadler} .

Clearly, we may assume that $\emptyset\notin \G$ and thus all sets in $\F$ have nonempty interiors. It remains to check that the elements of $\F$ have pairwise disjoint interiors. Firstly, consider arbitrary $G\in\G$ and any open $U\subseteq X$ such that $G\cap U=\emptyset$. Then $\cl{G}\cap U=\emptyset$ since $U$ is open, and thus, in particular, $\int\cl{G}\cap U=\emptyset$. Finally, let $G_1\ne G_2\in\G$, then $G_1\cap G_2=\emptyset$. Hence, using the previous part for $G=G_1$ and $U=G_2$, we obtain $\int\cl{G_1}\cap G_2=\emptyset$. Thus, we can use the same part again for $G=G_2$ and $U=\int\cl{G_1}$, obtaining $\int\cl{G_1}\cap \int\cl{G_2}=\emptyset$, which concludes the proof.
\end{proof}

\begin{lemma} \label{intermediatethrPeano}
    Let $X$ be a Peano continuum, $f:X\to \R$ continuous and $I\subseteq f(X)$ a closed interval. Then there exists a Peano continuum $Y\subseteq X$ satisfying $f(Y)=I$. 
\end{lemma}

\begin{proof}
Let $I=[a,b]$. Since $a,b\in f(X)$, there are $u,v\in X$ such that $f(u)=a$, $f(v)=b$. Since Peano continua are arcwise connected (see \cite[Theorem 8.23]{nadler}), there is an arc $A\subseteq X$ with endpoints $u,v$. Since $A$ is connected, so is $f(A)$. Moreover, since $a,b\in f(A)$, we get $[a,b]\subseteq f(A)$. It is now easy to find a suitable subarc $Y\subseteq A$ satisfying $f(Y)=I$.
\end{proof}

\subsection{Dynamical systems}
A dynamical system is a pair $(X,f)$, where $X$ is a compact metrizable topological space and $f:X\to X$ is a continuous map. As usual, we denote $f^0=id_X$ and $f^{n+1}=f\circ f^n$ for $n\geq 0$.
An \emph{$f$-orbit} of a point $x\in X$ is the sequence $(x,f(x),f^2(x),\dots)$.
By $\C(X)$ we mean the space of all maps from $X$ to itself equipped with the compact-open topology, which is, in fact, induced by the supremum metric.
This makes $\C(X)$ a complete separable metric space.

We say that $f$ is transitive if for every $U$, $V$ nonempty open subsets of $X$ there exists a natural number $n$ such that $f^n(U)$ intersects $V$ (or equivalently for a space without isolated points, there exists a point $x \in X$ such that the set $\{x, f(x), f^2(x), \dots\}$ is dense in $X$).
See e.g. \cite[Theorem 2.8]{AkinAuslanderNagar} for other characterizations of transitive maps.
We say that $f$ is mixing if for every pair of nonempty open sets $U,\,V\subseteq X$ there exists $N \in\N$ such that for every $n\geq N$, the set $f^n(U)\cap V$ is nonempty.
We say that a continuous map $f:X \to X$ is exact or LEO (locally eventually onto), if for every nonempty open $U \subseteq X$ there exists $n \in \N$ such that $f^n(U) = X$. 
We say that $f$ is Devaney chaotic if it is transitive and the set of periodic points of $f$ is dense in $X$. 
Following the terminology of \cite{KwietniakMisiurewicz}, we say that $f$ is exact Devaney chaotic if it is exact and Devaney chaotic. It is known that both Devaney chaotic maps and maps with positive topological entropy are Li-Yorke chaotic \cite{HuangYe, BlanchardGlasnerKolyadaMaass}. However, the first two notions are not related in general.

%chains - stuff

We say that $x_0,x_1,\dots,x_n\in X$ is an $\varepsilon$-chain (of $f$) from $x_0$ to $x_n$ of length $n$ if $d(f(x_i),x_{i+1})<\varepsilon$ for every $0\leq i\leq n-1$.
A map $f$ is called \emph{chain recurrent} if for every point $x\in X$ and $\delta>0$ there is a $\delta$-chain from $x$ to $x$.
A map $f$ is called \emph{chain transitive} if for every pair of points $x, y\in X$ and $\delta>0$ there is a $\delta$-chain from $x$ to $y$.
We say that $f$ is \emph{chain-mixing} if for every $\delta>0$ there exists $N\in\N$ such that for every $n\geq N$ and for every $x,y\in X$ there exists an $\delta$-chain from $x$ to $y$ of length exactly $n$.
In fact, if $X$ is connected, then chain recurrence is equivalent to both chain transitivity and chain mixing \cite[Corollary 14]{RichesonWiseman}. Moreover, $f$ is chain transitive iff no proper nonempty closed set $K\subseteq X$ satisfies $f(K)\subseteq \int(K)$
\cite[Theorem 4.12(4)]{Akin}.

% shadowing
A $\delta$-pseudo orbit is a sequence $(x_i)_{i\geq 0}$ such that $x_0,\dots, x_n$ is a $\delta$-chain for every $n\in\N$.
We say that $f$ has \emph{shadowing} or \emph{the shadowing property} if for every $\varepsilon>0$ there exists $\delta >0$ such that for every $\delta$-pseudo orbit $x_0,x_1, \dots$ there exists an $\varepsilon$-close  orbit, i.e. there exists $x\in X$ such that $d(x_i, f^i(x))<\varepsilon$ for every $i=0,1,\dots$.
By compactness of $X$, $f$ has shadowing iff for every $\varepsilon>0$ there is $\delta>0$ such that for every $\delta$-chain $(x_0,\dots, x_n)$ there is $x\in X$ such that $d(x_i, f^i(x))<\varepsilon$ for every $i=0,\dots,n$.

%Classes of maps
We denote as $\T(X), \CT(X), \LEO(X)$ and $\M(X)$ the collections of all transitive, chain transitive, locally eventually onto, and mixing maps, respectively. 
Furthermore, we denote by $\SH(X)$ and $\DP(X)$ the collections of maps with shadowing and with a dense set of periodic points, respectively. It is easy to observe that
\[\LEO(X)\subseteq \M(X)\subseteq \T(X)\subseteq \CT(X)=\cl{\CT(X)}.\]

Consequently, $\CT(X)$ is a complete space, so it makes sense to talk about a typical map in $\CT(X)$.

\subsection{The class $\LC$}

Let us denote by $\LC(X,Y)$ the collection of all $f \in \C(X,Y)$ that are locally constant on a dense (open) subset of $X$; that is, $f \in \LC(X,Y)$ if and only if there exists a dense set $D \subseteq X$ such that for every $x \in D$, there is a neighborhood $U$ of $x$ with $f|_U$ constant.

Clearly, $\LC(X,Y)$ contains all constant functions, but it also includes, for example, the Devil’s staircase function (also known as the Cantor ternary function) when $X = Y = [0,1]$. If $X = Y$, we simply write $\LC(X)$ instead of $\LC(X,Y)$.

In this subsection, we begin with a simple yet useful generalization of \cite[Theorem 8.19]{nadler}, and proceed to prove that if $\LC(X)$ is dense in $\C(X)$, then $\CT(X)$ is nowhere dense in $\C(X)$. We conclude by showing that $\LC(X)$ is dense in $\C(X)$ for a broad class of spaces—namely, all one-dimensional Peano continua and all absolute neighborhood retracts (compare with Question~\ref{otazkaPeano}).

\begin{lemma}\label{LCXY}
Let $X, Y$ be nondegenerate Peano continua, $K\subseteq X$ nowhere dense closed set, $x_1,\dots x_m\in X\setminus K$ distinct and $y_0,y_1,\dots,y_m\in Y$. Then there exists a surjective continuous map $f\in \LC(X,Y)$ satisfying $f(K)\subseteq \{y_0\}$, $f(x_1)=y_1$, $\dots$, $f(x_m)=y_m$.
\end{lemma}

\begin{proof}
Let $L_n\subseteq X\setminus K$ be a null sequence of disjoint continua such that the union of the interiors of $L_n$ is dense in $X$ and $x_i\in L_i$ for every $1\leq i\leq m$.
Consider the quotient space $Z$ and the quotient map $q: X\to Z$ such that each $L_n$ and also the set $K$ are pushed to a point (if $K=\emptyset$, we let $Z=X$, and $q$ is the identity). Note that $Z$ is a nondegenerate Peano continuum since the quotient is obtained by an upper semi-continuous decomposition. Let $h: Z\to Y$ be a continuous surjective map satisfying $h(q(K))=y_0$ if $K\ne\emptyset$ and $h(q(L_1))=y_1,\dots,h(q(L_m))=y_m$ \cite[8.19]{nadler}. It is easy to see that $f:=h\circ q$ has the desired properties.
\end{proof} 

\begin{proposition}\label{CTnowheredense}
    Let $X$ be a continuum satisfying that $\LC(X)$ is dense in $\C(X)$. Then $\CT(X)$ is nowhere dense in $\C(X)$.
\end{proposition}

\begin{proof}
    Since $\CT(X)$ is closed in $\C(X)$, we only need to check that $\C(X)\setminus\CT(X)$ is dense in $\C(X)$. We may assume that the metric $X$ is equipped with is convex since $X$ admits a compatible convex metric by \cite[Theorem 8]{Bing}. Hence let $f'\in\C(X)$ and $\varepsilon>0$. Since we assume that $\LC(X)$ is dense in $\C(X)$, there is $f\in\LC(X)$ such that $d(f,f')<\varepsilon/2$. We want to find $g\in \C(X)\setminus\CT(X)$ satisfying $d(f,g)<\varepsilon/2$. If $f\notin \CT(X)$ then we may put $g:=f$, thus assume that $f\in\CT(X)$.

    Fix some nonempty open $U$ where $f$ is constant, and pick any $x\in U$. Since $f\in\CT(X)$, we can find an $\varepsilon/2$-chain  $x_0,\dots,x_n$ such that $x_0= x_n=x$. We may assume that the points $x_0,\dots,x_{n-1}$ are distinct. Fix Peano continua $K_0,L_n$ such that $x\in\int L_n\subseteq L_n\subseteq\int K_0\subseteq K_0\subseteq U\setminus\{x_1,\dots,x_{n-1}\}$. For each $1\leq i\leq n-1$ fix a Peano continuum $K_i\subseteq B(x_i,\varepsilon/2)$ with nonempty interior such that $f$ is constant on each $K_i$ and the sets $K_0,K_1,\dots,K_{n-1}$ are pairwise disjoint. Further, for each $1\leq i\leq n-1$ fix a point $y_i\in\int K_i$. Also put $y_n:=x_n$.

    We are prepared now to define $g$. Let $g|_{X\setminus (K_0\cup\dots\cup K_{n-1})}:=f|_{X\setminus (K_0\cup\dots\cup K_{n-1})}$. Further, note that for each $0\leq i\leq n-1$ there exists a Peano continuum $Y_i$ such that $f(\bd K_i)\cup\{y_{i+1}\}\subseteq Y_i\subseteq B(x_{i+1},\varepsilon/2)$ since the metric we work with is convex. Hence for each $1\leq i\leq n-1$ we can find $g|_{K_i}:=K_i\to Y_i$ such that $g(\bd K_i)=f(\bd K_i)$ (a singleton) and $g(y_i):=y_{i+1}$ by Lemma \ref{LCXY}. Finally, let $g|_{K_0}:K_0\to Y_0$ be any continuous map satisfying $g(\bd K_0)=f(U)$ and $g(L_n)=\{y_1\}$, there is such a map by, i.e., a straightforward generalization of Theorem \cite[8.19]{nadler} or by a natural modification of Lemma \ref{LCXY}. 

    It is easy to see that the obtained map $g$ is well-defined and continuous. Further, $d(f',g)\leq d(f',f)+d(f,g)<\varepsilon/2+\varepsilon/2=\varepsilon$. Note that $g^n(x)=x$, and $g$ is constant on $L_n$, a closed neighborhood of $x$. It is easy to find inductively closed sets $L_{n-1}$, $L_{n-2}$, $\dots$, $L_0$ such that $g(L_{i})\subseteq \int L_{i+1}$ and $g^i(x)\in\int L_i$ for $i=n-1,n-2,\dots,0$. Let $L:=L_0\cup\dots\cup L_n$. Then $L$ is a nonempty closed proper subset of $X$ satisfying $f(L)\subseteq\int L$. Hence $f\notin \CT(X)$.
\end{proof}

\begin{lemma}\label{extension->LC}
Let $X,Y$ be continua such that
for every $y\in Y$ and every neighborhood $V$ of $y$ there is a neighborhood  $L\subseteq V$ of $y$, such that every continuous map $f:K\to L$ from a closed set $K\subseteq X$ has a continuous extension $\bar{f}:X\to V$.
Then $\LC(X,Y)$ is dense in $\C(X,Y)$.
\end{lemma}

\begin{proof}
Let $f\in\C(X,Y)$ be arbitrary and $\varepsilon>0$. Let $\varepsilon_n:=2^{-n}\varepsilon$. Let $x_n$, $n\in\N$, satisfy that $\{x_n,\,n\in\N\}$ is dense in $X$. We will find inductively maps $g_n\in \C(X,Y)$, $n\geq 0$, and open $U_n\subseteq X$, $n\in\N$, satisfying:
    
\begin{itemize}
        \item $x_n\in \cl{U_1}\cup\dots\cup \cl{U_n}$,
        \item $U_1,\dots, U_n$ are pairwise disjoint,
        \item $g_n|_{U_i}=g_{n-1}|_{U_i}$ for every $1\leq i <n$,
        \item $g_n|_{U_n}$ is constant,
        \item $d(g_n,g_{n-1})\leq \varepsilon_n$.
    \end{itemize}
    
    Let $g_0=f$. Assume that we already have $g_n, U_1,\dots, U_n$. Let $m\in\N$ be the least natural number satisfying $x_m\notin \cl{U_1}\cup\dots\cup \cl{U_n}$. By our assumption, there is $L\subseteq B(f(x_m),\varepsilon_n/2)$ a neighborhood of $f(x_m)$ such that partial maps from $X$ into $L$ can be continuously extended to maps into $B(f(x_m),\varepsilon_n/2)$. Clearly, there is a closed $K'\subseteq X\setminus(\cl{U_1}\cup\dots\cup \cl{U_n})$  satisfying $x_m\in \int K'$ and $g_n(K')\subseteq L$. Fix some open $U_{n+1}\subseteq X$ satisfying $x_m\in U_{n+1}\subseteq\cl{U_{n+1}}\subseteq\int K'$.

     By the assumptions of this Lemma (used with $K:=\cl{U_{n+1}}\cup\bd K'$) there is $g_{n+1}\in\C(X,Y)$ satisfying

\begin{itemize}
\item $g_{n+1}$ equals to $g_{n}$ on $\bd K'$ and thus we may assume that $g_{n+1}$ equals to $g_{n}$ on the set $X\setminus \int K'\supseteq U_1\cup\dots\cup U_n$,
\item $g_{n+1}$ is constant on $U_{n+1}$, an open neighborhood of $x_m$,
\item $g_{n+1}(K')\subseteq B(f(x_m),\varepsilon_n/2)$, hence in particular, $d(g_n,g_{n+1})\leq \varepsilon_n$.
\end{itemize}

Consequently, the sequence $g_n$ converges uniformly to some $g\in\C(X,Y)$ and the distance of $g$ and $f=g_0$ is less than $\sum\varepsilon_n=\varepsilon$.
Since $g_k$ is constant on $U_n$ for every $k\geq n$, it follows that $g$ is locally constant on the open set $U:=\bigcup_{n\in\N} U_n$. Moreover, $U$ is dense since
\[
X=\cl{\{x_n,\,n\in\N\}}\subseteq \bigcup_{n\in\N} \cl{U_n}\subseteq \cl{\bigcup_{n\in\N} U_n}=\cl{U}
\]

\end{proof}

To verify the assumptions of Lemma \ref{extension->LC} in specific cases, we will make use of the following extension result; see, for example, \cite[Theorem 6.2]{MayerOversteegenTymchatyn} or \cite[Theorem 4.2.31]{vanMill} with $n=0$.

\begin{proposition}[Dugundji, Kuratowski]\label{prop:DugundjiKuratowski}
Let $X$ be a metric space, $Y$ a Peano continuum, $K\subset X$ closed, $\dim(X\setminus K)\leq 1$. 
Then every continuous function $f\colon K\to Y$ has a continuous extension $\Bar{f}\colon X\to Y$.
\end{proposition}

\begin{theorem}\label{LCpodminka}
    Let $X$ be a Peano continuum that is either one-dimensional or an absolute neighborhood retract. Then $\LC(X)$ is dense in $\C(X)$.
\end{theorem}

\begin{proof}
In both cases, we will verify the assumptions of Lemma \ref{extension->LC}. Let $x\in X$ and $V$ be a neighborhood of $x$.
    
    First, assume that $X$ is one-dimensional. There is a Peano continuum $L\subseteq V$ such that $x\in\int L$ by Proposition \ref{partitions}. Let $K\subseteq X$ be closed and $f:K\to L$ continuous. Since $\dim (X\setminus K)\leq\dim X= 1$, there is $\bar{f}:X\to L$ a continuous extension of $f$ by Proposition \ref{prop:DugundjiKuratowski}.
    
    Secondly, assume that $X$ is an absolute neighborhood retract. Denote by $Q$ the Hilbert cube, i.e. the space $[0,1]^\omega$. We may assume that $X\subseteq Q$. Since $X$ is an absolute neighborhood retract, there is a neighborhood $U$ of $X$ in $Q$ and a retraction $r:U\to X$. Find $R\subseteq U$ a neighborhood of $x$ in $Q$ that is homeomorphic to $Q$ such that $r(R)\subseteq V$ (note that, indeed, $r(x)=x$ as $x\in X$). Find $W\subseteq R$ an open set (in $Q$) containing $x$ such that $r(W)\subseteq R$. Note that $r(W)$ is a neighborhood of $x$ in $X$ since it contains the open set $W\cap X$.
    
    We claim that $L:=r(W)$ has the required extension property. To verify this, let $K\subseteq X$ be closed and $f:K\to L$ continuous. Since $L=r(W)\subseteq R$, and $R$ is homeomorphic to $Q$, there is $\bar{f}:X\to R$ a continuous extension of $f$. Observe that $r\circ \bar{f}:X\to r(R)\subseteq V$ is continuous and for every $z\in K$ we have $r(\bar{f}(z))=r(f(z))$ by the choice of $\bar{f}$ and $r(f(z))=f(z)$ since $f(z)\in L=r(W)\subseteq X$.
\end{proof}

\section{Chaos on Peano continua}

\begin{lemma}\label{prop}
    Let $X$ be a Peano continuum and $f\in\LC(X)$. Suppose that $\F$ is a finite cover of $X$ formed by Peano subcontinua of $X$ with nonempty pairwise disjoint interiors such that there exists $n_0$ satisfying $f^{n_0}(F)=X$ for every $F\in\F$. Then there exists $g\in \LEO(X)\cap\DP(X)\cap \cl{\LC(X)}$ such that $d(f,g)\leq 2\cdot \mesh\F$.
\end{lemma}

\begin{proof}
    Put $\varepsilon:=2\cdot \mesh\F$. We may assume $\varepsilon\in (0,\infty)$ since $\LC(X)\ne\emptyset$. 
    We will construct by induction with respect to $n\geq 0$:
\begin{itemize}
    \item $\F_n$: a finite cover of $X$ formed by Peano subcontinua of $X$ with nonempty pairwise disjoint interiors,
    \item $g_n\in\LC (X)$,
    \item $S_n$: a finite subset of $X$,
\end{itemize}

such that:
\begin{enumerate}
    \item\label{prop1} $d(g_n,g_{n+1})\leq\varepsilon \cdot 2^{-n}$,
    \item\label{prop2} $g_n^{i+n_0}(F)=X$ for every $F\in \F_i$ and $0\leq i\leq n$,
    \item\label{prop3} $\mesh\F_n\leq\varepsilon\cdot 2^{-n-1}$,
    \item\label{prop4} $\F_{n+1}$ refines $\F_n$,
    \item\label{prop5} $g_n(S_n)=S_n$ (and thus points in $S_n$ are periodic under $g_n$),
    \item\label{prop6} $S_n\subseteq S_{n+1}$ and $g_{n+1}|_{S_n}=g_n$,
    \item\label{prop7} if $n\geq 1$, then $S_n\cap F\ne \emptyset$ for every $F\in \F_n$.
\end{enumerate}

Let $\F_0:=\F$, $g_0:=f$ and $S_0:=\emptyset$; clearly these satisfy the induction hypotheses. Assume that we have already defined $S_n, g_n$, $\F_n$ and we will find $S_{n+1},g_{n+1}$, $\F_{n+1}$. There exists $0<\delta<\varepsilon\cdot 2^{-n-2}$ such that whenever $A\subseteq X$ is of diameter $<\delta$ then $g_n(A)$ is of diameter $<\varepsilon\cdot 2^{-n-1}$. By Proposition \ref{partitions} we can find $\F_{n+1}$ a finite cover of $X$ formed by Peano subcontinua of $X$ of diameter $<\delta$ with nonempty disjoint interiors refining $\F_n$.    

For every $F\in\F_{n+1}$ choose a point $x_F\in int (F)\setminus S_n$ such that $g_n$ is constant on a neighborhood of $x_F$ and $H_F\in\F_n$ satisfying $g_n(x_F)\in H_F$. By the inductive hypotheses we have $g_n^{n+n_0}(H_F)=X$ and thus there is $x_F'\in H_F$ such that $g_n^{n+n_0}(x_F')=x_F$. Put 
\[
S_{n+1}:= S_n\cup \{g_n^i(x_F');\,0\leq i\leq n+n_0,F\in\F_{n+1}\}.
\]

For every $F\in\F_{n+1}$ fix some Peano continuum $F'$ such that 
\[
x_F\in \int F'\subseteq F'\subseteq (\int (F)\setminus S_{n+1})\cup\{x_F\}
\]
and $g_n$ is constant on $F'$. Finally, let $g_{n+1}|_{F'}\in \LC (F', g_n(F)\cup H_F)$ be any surjective map satisfying $g_{n+1}(bd(F'))=g_n(bd(F'))$ (a singleton) and $g_{n+1}(x_F)=x_F'$, there exists some by Lemma \ref{LCXY}. Let $g_{n+1}(x):=g_n(x)$ for every $x\in X\setminus \bigcup \{F';\,F\in \F_{n+1}\}$.

To justify that the inductive hypotheses are indeed satisfied, we will only verify properties \ref{prop1} and \ref{prop2}, as checking the other hypotheses is straightforward. To verify \ref{prop1}, it suffices to observe that
\[
d(g_n(x),g_{n+1}(x))\leq \diam g_n(F) + \diam H_F\leq \varepsilon\cdot 2^{-n-1} + \varepsilon\cdot 2^{-n-1} = \varepsilon\cdot 2^{-n}
\]
if $x\in F'$ for some $F\in\F_{n+1}$, and $g_n(x)=g_{n+1}(x)$ otherwise. To verify \ref{prop2}, first fix $0\leq i\leq n$ and $F\in \F_i$. There is $\H\subseteq \F_{n+1}$ such that $F=\bigcup \H$ since $\{\F_i\}_{0\leq i \leq n+1}$ is a refining sequence of covers by continua with pairwise disjoint interiors. Thus, as $g_{n+1}(H)\supseteq g_{n}(H)$ for every $H\in \F_{n+1}$ by construction, we obtain that
\begin{equation}\label{magnifying}
    g_{n+1}(F) = g_{n+1}\left(\bigcup \H\right) =\bigcup \{g_{n+1}(H);H\in\H\} \supseteq \bigcup \{g_{n}(H);H\in\H\}  = g_n(F).
\end{equation}

Let $0\leq i\leq n+1$ and $F\in \F_i$. To prove that $g_{n+1}^{i+n_0}(F)=X$, we will distinguish cases $i\leq n$ and $i=n+1$. If $i\leq n$, then (\ref{magnifying}) and inductive hypotheses give $g_{n+1}^{i+n_0}(F) \supseteq g_n^{i+n_0}(F)=X$. Suppose that $i=n+1$ and recall that we have $g_{n+1}(F)\supseteq H_F$, where $H_F\in\F_n$ has been established during the construction. Therefore
\[
g_{n+1}^{n+1+n_0}(F) = g_{n+1}^{n+n_0}(g_{n+1}(F)) \supseteq g_{n+1}^{n+n_0}(H_F) \supseteq g_n^{n+n_0}(H_F)=X,
\]
where the last inequality, resp. the last equality, follows by (\ref{magnifying}), resp. by the inductive hypotheses.
Let $g$ be the limit of $g_n$, clearly $g\in\cl{\LC(X)}$ as $g_n\in\LC(X)$ for every $n$. By \ref{prop1} $g$ is well-defined and continuous, and moreover $d(g,g_0)\leq \varepsilon$. Thus $g_0=f$ entails $d(f,g)\leq\varepsilon=2\cdot \mesh\F$. By \ref{prop3} and \ref{prop7}, $\bigcup_{n\in\N} S_n$ is dense in $X$. By \ref{prop5} and \ref{prop6}, every point of $\bigcup_{n\in\N} S_n$ is a periodic point of $g$. Therefore $g\in \DP(X)$.
Let $ U\subseteq X$ be a nonempty open set, there exists $n\in \N$ and $F\in\F_n$ such that $F\subseteq U$. By \ref{prop2}, $g_k^{n+n_0}(F)=X$ for every $k\geq n$ and therefore $g^{n+n_0}(F)=X$ by the compactness of $F$. Thus, $g^{n+n_0}(U)\supseteq g^{n+n_0}(F)=X$. Hence $g\in \LEO (X)$, which conludes the proof.
\end{proof}

\begin{theorem}\label{existenceofexactDevaney}
Every Peano continuum admits a LEO selfmap with a dense set of periodic points.
\end{theorem}

\begin{proof}
    Let $X$ be a Peano continuum. By Lemma \ref{LCXY} there is a locally constant surjective map $f:X\to X$. By Lemma \ref{prop} (use the cover $\F:=\{X\}$ with $n_0:=1$) the continuum $X$ admits a LEO selfmap with a dense set of periodic points.
\end{proof}

\begin{theorem}\label{LeoDPdense}
     For every Peano continuum $X$ it holds that $\LEO(X)\cap \DP(X) \cap\cl{(\LC(X))}$ is dense $\CT(X)\cap \cl{(\LC(X))}$.
    \end{theorem}

\begin{proof}
    We may assume that $X$ is nondegenerate. Fix $d$ some compatible convex metric on $X$, there is some by \cite[Theorem 8]{Bing}. Let $f\in \CT(X)\cap \cl{(\LC(X))}$ and $\varepsilon>0$, we will find $g\in \LEO(X)\cap \DP(X)\cap\cl{(\LC(X))}$ satisfying $d(f,g)<\varepsilon$. Since $f\in \cl{(\LC(X))}$, there is $g_0\in \LC(X)$ such that $d(f,g_0)<\varepsilon/6$. There exists $0<\delta<\varepsilon/6$ such that whenever $A\subseteq X$ is of diameter $<\delta$ then $f(A)$ is of diameter $<\varepsilon/6$.
    
    By Proposition \ref{partitions} there is $\F$ a finite cover of $X$ formed by Peano subcontinua of $X$ with nonempty disjoint interiors satisfying $\mesh \F<\delta$. For every $F\in\F$, let \[Y_F = \bigcup \{H\in\F;\,d(f(F),H)<\varepsilon/6\},\] note that this set is connected since $d$ is convex, and hence a Peano continuum. Fix some Peano continuum $F'$ with a nonempty interior such that $F'\subseteq \int (F)$, and $g_0$ is constant on $F'$. Let $g_1|_{F'}\in\LC(F', Y_F)$ be any surjective map satisfying $g_1(bd(F')=g_0(bd(F')$ (a singleton), there exists some by Lemma \ref{LCXY} as $g_0(bd(F'))\subset Y_F$. Let $g_1(x):=g_0(x)$ for every $x\in X\setminus \bigcup \{F';\,F\in \F\}$. Clearly $g_1\in \LC(X)$ and moreover $d(f,g_1)<\varepsilon/2$ since for every $x\in X$ we have either
    \[d(f(x), g_1(x))\leq \diam f(F)+\varepsilon/6+\mesh \F\leq \varepsilon/6+\varepsilon/6+\delta\leq \varepsilon/2\]
    if $x\in F'$ for some $F\in \F$, or $g_1(x)=g_0(x)$ together with $d(f,g_0)<\varepsilon/6$ if no such $F\in\F$ exists.
    
    We claim that there exists $n_0 \in \N$ such that $g_1^{n_0}(F)=X$ for every $F\in \F$. Since $f\in\CT(X)$, by \cite[Corollary 14]{RichesonWiseman} there is $n_0\in\N$, such that for every $n \in\N$, $n\geq n_0$ and for every $x,y\in X$ there exists an $\varepsilon/6$-chain from $x$ to $y$ of length exactly $n$. Fix $F,H\in \F$. We can find  an $\varepsilon/6$-chain $x_0,\dots,x_{n_0}$ such that $x_0\in F$ and $x_{n_0}\in H$. Put $F_0:=F$, $F_{n_0}:=H$ and for every $2\leq i\leq n_0-1$ fix any $F_i\in\F$ satisfying $x_i\in F_i$. If $1\leq i\leq n_0$, observe that $F_i\subseteq Y_{F_{i-1}} =g_1(F_{i-1})$ since
    \[
    d(F_i,f(F_{i-1}))\leq d(x_i,f(x_{i-1}))<\varepsilon/6.
    \]

     It follows easily by induction that $g_1^{n_0}(F_0)\supseteq F_{n_0}$, in other words, $g_1^{n_0}(F)\supseteq H$. Since $F,H\in \F$ were arbitrary and $\F$ is a cover, we obtain that $g_1^{n_0}(F)=X$ for every $F\in\F$. Thus by Lemma \ref{prop} there exists $g\in \LEO(X)\cap\DP(X)\cap\cl{\LC(X)}$ such that $d(g,g_1)\leq2\cdot\mesh \F<2\cdot \varepsilon/6<\varepsilon/2$, hence $d(f,g)\leq d(f,g_1)+d(g_1,g)<\varepsilon$.

\end{proof}

\section{Mixing is generic among chain transitive maps}

Throughout the whole section, we will assume that $X$ is a fixed nondegenerate Peano continuum and $d$ is a fixed compatible convex metric on $X$ (every Peano continuum admits such a metric by \cite[Theorem 8]{Bing}). Further, we fix $\{\H_n\}_{n\in\N}$ a refining sequence of finite covers of $X$ formed by Peano continua with pairwise disjoint nonempty interiors such that $\mesh \H_n< 2^{-n}$ for every $n\in\N$, such a sequence can be constructed easily by the inductive usage of Proposition \ref{partitions}. Thus, under the described setup, we denote for $n\in\N$
\[
G_n:=\{f\in \C(X);\,\forall F,H\in\H_n \exists k_0 \forall k\geq k_0:\,f^k(F)\cap H \ne \emptyset \}.
\]

Note that $G_m\subseteq G_k$ whenever $m\geq k$ since $\{\H_n\}_{n\in\N}$ is a  refining sequence of covers.

\begin{lemma} \label{Gisgeneric}
    For every $n\in\N$ it holds that $\int (G_n\cap \CT(X)\cap \cl{\LC(X)})$ is dense in  $\CT(X)\cap \cl{LC(X)}$, where the interior is taken with respect to the subspace topology on $\CT (X)\cap \cl{\LC(X)}$.
\end{lemma}

\begin{proof}
Let $n\in \N$, $f\in \CT(X)\cap \cl{\LC(X)}$ and $\varepsilon>0$. We will find $g\in \C(X)$, $\xi>0$ such that
\[
\emptyset\ne B(g,\xi)\cap\CT(X)\cap \cl{\LC(X)} \subseteq G_n\cap \CT(X)\cap \cl{\LC(X)}\cap B(f,\varepsilon).
\]
There exists $0<\delta<\min\{\varepsilon/4, \diam (X)\}$ such that if $A\subseteq X$, $\diam A<\delta$, then $\diam f(A)<\varepsilon/4$. By possibly making  $n$ larger we may assume that $2^{-n}<\delta$. There is $h\in \LC(X)$ satisfying $d(f,h)<\varepsilon/4$.

For every $H\in \H_n$ fix a nonempty open $U_H\subseteq H$ such that $h$ is constant on $U_H$ and a point $x_H \in U_H$. After repeating this process for every $H\in\H_n$, we can fix $0<\xi<\varepsilon/4$ satisfying that $B(x_H,4\xi)\subseteq U_H$ for every $H\in \H_n$. For every $H\in \H_n$ fix a point $y_H\in U_H$ satisfying $d(x_H,y_H)= 3\xi$. By Proposition \ref{partitions} there is $\F$ a finite cover of $X$ refining $\H_n$ formed by Peano continua with pairwise disjoint nonempty interiors such that $\mesh \F< \xi/2$.

We will proceed with the construction independently for each $H\in\H_n$, so fix $H\in\H_n$. Let 
\[\H_H := \{F\in\H_n;\,d(f(H),F)<\varepsilon/4\}.\]   
For every $F\in\F$, $F\subseteq H\setminus U_H$, fix a Peano continuum $P_F\subseteq\int F$ with nonempty interior such that $h$ is constant on $P_F$. By Lemma \ref{LCXY} there is $p_F\in \LC (P_F,\bigcup H_H)$ that is surjective and satisfies $p_F(\bd (P_F))=h(P_F)$ (a singleton).
 Fix (necessarily distinct) points $z_F\in \int F\cap U_H$ for every $F\in F$ satisfying $ F\cap U_H\ne\emptyset$. Let $s_H:\cl{U_H} \to [0,3]$ be any continuous map such that

\begin{itemize}
    \item $s_H(B(x_H,\xi))\subseteq [2,3]$,
    \item $s_H(B(y_H,\xi))\subseteq [0,1]$,
   \item  $s_H(\bd U_H)=\{0\}$,
   \item the points $s_H(z_F)$, $F\in \F$, $F\cap U_H\ne\emptyset$, are distinct,
   \item $|s_H(W)|>1$ for every $W$ a neighborhood of $z_F$ for every $F\in F$ satisfying  $F\cap U_H\ne\emptyset$.
\end{itemize}
   It is easy to see that a map satisfying all but the last condition exists using the Tietze extension theorem. The resulting map can be fixed by being altered on small neighborhoods of $z_F$ using the map $d(-,z_F)$, if necessary.
    
Let $t_H\in\LC([0,3], \bigcup \H_H)$ be any continuous map satisfying
\begin{itemize}
    \item $\{t_H(0)\}=h(U_H)$,
    \item for every $H'\in\H_H$ there is an interval $I_{H,H'}\subseteq [1,2]$ such that $x_{H'},y_{H'} \in t_{H}(I_{H,H'})\subseteq \cl{B(x_{H'},3\xi)}$,
    \item for every $F\in \F$  satisfying  $ F\cap U_H\ne\emptyset$, there is an interval $J_{F}\subseteq [0,3]$ such that $s_H(F)\supseteq J_F$ and $t_H(J_F)\supseteq \bigcup \H_H$.
\end{itemize}
Note that this is possible since $d$ is convex and $h(U_H)\in \bigcup \H_H$ as $d(f,h)<\varepsilon/4$; to see that it is indeed possible to find $t_H$ in $\LC$ it may be convenient to use Lemma \ref{LCXY}. After repeating the process for every $H\in\H_n$ we can finally define $g:X\to X$:
     \begin{equation}
    g(x) :=
    \begin{cases}
      (t_H \circ s_H)(x) & \text{if } x\in \cl{U_H}, H\in\H_n, \\
      p_F(x)           & \textbf{if }  x \in P_F, F\in \F,F\subseteq H\setminus U_H, H\in\H_n, \\
      h(x)        & \text{otherwise.} 
    \end{cases}
  \end{equation}
     
Clearly $g$ is well-defined and it is continuous since for every $H\in\H_n$ and every $F\in\F$, $F\subseteq H\setminus U_H$, we have $(t_H \circ s_H)(\bd U_H)=t_H(0)=h(U_H)$ and $p_F(\bd (P_F))=h(P_F)$ by our choices. We will prove that $g$, $\xi$ have the required properties in a series of claims.

\begin{claim*}
    $B(g,\xi) \subseteq  B(f,\varepsilon)$.
\end{claim*}

\begin{proofofclaim}
    Let $x\in X$, we find $H\in\H_n$ satisfying $x\in H$. By the construction $g(x)\in\bigcup \H_H$, therefore  the definition of $\H_H$ and $f(x)\in f(H)$ give
    \[d(f(x), g(x))\leq \diam f(H)+\varepsilon/4+\mesh \H_n\leq \varepsilon/4+\varepsilon/4+\delta\leq 3\varepsilon/4.\]

     Thus $\xi<\varepsilon/4$ concludes the proof.
\end{proofofclaim}

Since $f\in\CT(X)$, by \cite[Corollary 14]{RichesonWiseman} there is $n_0\in\N$ such that for every $n \in\N$, $n\geq n_0$, and for every $x,y\in X$ there exists an $\varepsilon/4$-chain from $x$ to $y$ of length exactly $n$.
    
\begin{claim*}
    $\emptyset\ne B(g,\xi)\cap\CT(X)\cap \cl{\LC(X)}.$
\end{claim*}

\begin{proofofclaim}
    First, we show that $g\in \LC(X)$. Let $U\subseteq U_H$ for some $H\in\H_n$ be nonempty and open, by possibly making it smaller we may assume that $U$ is connected since $X$ is Peano. If $|s_H(U)|=1$ then clearly $g$ is constant on $U$. Otherwise it is a nondegenerate interval and thus $t_H$ is constant on an open set $V\subseteq s_H(U)$. Thus $g$ is constant on an open nonempty set $s_H^{-1}(V)\cap U$. Hence we conclude that $g\in \LC(X)$ since checking the other cases is straightforward.

    Fix $F,H\in \F$. We can find  an $\varepsilon/4$-chain $x_0,\dots,x_{n_0}$ such that $x_0\in F$ and $x_{n_0}\in H$ by the choice of $n_0$. Put $F_0:=F$, $F_{n_0}:=H$ and for every $2\leq i\leq n-1$ fix any $F_i\in\F$ satisfying $x_i\in F_i$. If $1\leq i\leq n_0$, observe that $F_i\subseteq g(F_{i-1})$ since $d(F_i,f(F_{i-1}))\leq d(x_i,f(x_{i-1}))<\varepsilon/4$. It follows easily by induction that $g^{n_0}(F_0)\supseteq F_{n_0}$, in other words, $g^{n_0}(F)\supseteq H$. Since $F,H\in \F$ were arbitrary and $\F$ is a cover, we obtain that $g^{n_0}(F)=X$ for every $F\in\F$. Thus the hypotheses of Lemma \ref{prop} are satisfied and therefore $\LEO(X)\cap \DP(X)\cap \cl{\LC(X)}\cap B(g,2\mesh\F)\ne\emptyset$. In particular, $\CT(X)\cap \cl{\LC(X)}\cap B(g,\xi)\ne\emptyset$ as $\LEO(X)\subseteq\CT(X)$ and $\mesh\F<\xi/2$.
    
    \end{proofofclaim}

Finally, we aim to prove that $B(g,\xi) \subseteq G_n$. To this end, fix any map $h$ satisfying $d(g,h)<\xi$; we will show that $h\in G_n$. To proceed, fix $F,H\in \H_n$ and $i\in\N$, $i\geq n_0$.  Our goal is to prove that $h^i(F)\cap H\ne\emptyset$. Since $i\geq n_0$, and by the choice of $n_0$, we can find  $x_0,\dots,x_{i}$, an $\varepsilon/4$-chain of $f$, such that $x_0\in F$ and $x_{i}\in H$. Put $F_0:=F$, $F_{i}:=H$ and for every $2\leq k\leq i-1$ fix any $F_k\in\H_n$ satisfying $x_k\in F_k$. We will prove by induction with respect to $k$, $0\leq k\leq i$, the following key claim.

    \begin{claim*}
        For every $0\leq k\leq i$ there exists $L\subseteq F$ such that $h^k(L)\subseteq U_{F_k}$ is a Peano continuum intersecting both $B(x_{F_k},\xi)$, $B(y_{F_k},\xi)$.
    \end{claim*}

    \begin{proofofclaim}
    If $k=0$ the statement becomes trivial, we can choose e.g $L:=\cl{B(x_0,3\xi)} $. Assume that the claim is true for $k-1$ and thus, by the induction hypothesis, there exists $L'\subseteq F$ such that $h^{k-1}(L')\subseteq U_{F_{k-1}}$ is a Peano continuum intersecting both $B(x_{F_{k-1}},\xi)$, $B(y_{F_{k-1}},\xi)$. Let us examine $g(h^{k-1}(L'))$ first. Since $h^{k-1}(L')\subseteq U_{F_{k-1}}$, we have 
    \[
    g|_{h^{k-1}(L')}= (t_{F_{k-1}} \circ s_{F_{k-1}})|_{h^{k-1}(L')}.
    \]
    Note that $h^{k-1}(L')$ intersecting both $B(x_{F_{k-1}},\xi)$, $B(y_{F_{k-1}},\xi)$ gives that $s_{F_{k-1}}(h^{k-1}(L'))$ intersects both $[0,1]$, $[2,3]$, by the definition of $s_{F_{k-1}}$. Thus $[1,2]\subseteq s_{F_{k-1}}(h^{k-1}(L'))$ by the connectedness of $L'$. Observe that
    \[
    d(F_k,f(F_{k-1}))\leq d(x_k,f(x_{k-1}))<\varepsilon/4
    \]    
    and hence $F_k\in \H_{F_{k-1}}$. Thus by the definition of $t_{F_{k-1}}$, there is an interval $I_{F_{k-1},F_k}\subseteq [1,2]$ such that $x_{F_k},y_{F_k} \in t_{F_{k-1}}(I_{F_{k-1},F_k})\subseteq \cl{B(x_{F_k},3\xi)}$ .
     
     By Lemma \ref{intermediatethrPeano} there is $L''\subseteq h^{k-1}(L')$ a Peano continuum satisfying $s_{F_{k-1}}(L'')=I_{F_{k-1},F_k}$. Thus $x_{F_k},y_{F_k} \in g(L'')\subseteq \cl{B(x_{F_k},3\xi)}$ as $g(L'')=t_{F_{k-1}}(s_{F_{k-1}}(L''))=t_{F_{k-1}}(I_{F_{k-1},F_k})$. Hence $d(g,h)<\xi$ entails that $h(L'')\subseteq B(x_{F_k},4\xi)\subseteq U_{F_k}$ is a Peano continuum intersecting both $B(x_{F_k},\xi)$, $B(y_{F_k},\xi)$. Hence $L:=h^{-(k-1)}(L'')\cap F$ is as desired since $h^k(L)=h(L'')$. 
    \end{proofofclaim}
  
    The last claim easily gives that $h^i(F)$ and $H$ intersect, thus proving that indeed $h\in G_n$, which concludes the proof.
\end{proof}

\begin{theorem}\label{genericmixing}
    For every Peano continuum $X$, a generic map in $\CT(X)\cap \cl{\LC(X)}$ is mixing.
\end{theorem}

\begin{proof}
    It suffices to show that all maps in $G:=\bigcap_{n\in\N} G_n$ are mixing since $G\cap \CT(X)\cap \cl{\LC(X)}$ is comeager in $\CT(X)\cap \cl{\LC(X)}$ by Lemma \ref{Gisgeneric}. Let $f\in G$ and $U,V\subseteq X$ be nonempty and open. Then there exist $n\in\N$ and $F,H\in\H_n$ such that $F\subseteq U$ and $H\subseteq V$. Since $f\in G\subseteq G_n$, there exists $k_0$ such that for every $k\geq k_0$ the sets $f^k(F)$ and $H$ intersect. Therefore $f^k(U)$ and $V$ intersect for every $k\geq k_0$, proving that $f$ is mixing.
\end{proof}

\section{Shadowing is generic}

\begin{lemma} \label{opencoverslebesgue}
Let $X$ be a Peano continuum and $\delta>0$. 
Then there exists %$\lambda>0$ and 
a finite family $\F$ formed by Peano subcontinua of $X$, such that:
    \begin{itemize}[noitemsep]
    \item $\mesh(\F)<\delta$,
    \item $\{\int(F): F\in\F\}$ covers $X$,
%        \item for every $x\in X$ there exists $F\in\F$ satisfying $B(x,\lambda)\subseteq F$,
        \item for every $F\in\F$ the set $F\setminus \bigcup (\F \setminus \{F\})$ is nonempty (and open).
    \end{itemize}
\end{lemma}

\begin{proof}

Let $\H$ be a finite family of Peano subcontinua of $X$ with $\mesh\H<\delta/3$ which covers $X$ and such that $\H$ is minimal with respect to inclusion. Then for every $H\in\H$ there exists $x_H\in H\setminus \bigcup (\H\setminus \{H\})$.
Let $M=\bigcup \{\bd H:H\in\H\}$. Clearly, $M$ is a closed nowhere dense set which is disjoint with $\{x_H:H\in\H\}$.
By compactness of $M$, we may find a finite family $\mathcal K$ with

\begin{itemize}[noitemsep]
    \item $\mesh\mathcal K<\delta/3$,
    \item $\mathcal K$ is formed by Peano subcontinua of $X$ with nonempty interiors,
    \item $M\subseteq\bigcup \{\int K: K\in\mathcal K\}$,
    \item every $K\in\mathcal K$ is disjoint with the finite set $\{x_H:H\in\H\}$.
\end{itemize}
 For every $K\in\mathcal K$ there is some $H_K\in\H$ such that $K\cap H_K\neq\emptyset$. Let 
 \[H'=H\cup \bigcup \{K: H_K=H\}.\]
 
 Then the family $\F=\left\{H': H\in\H\right\}$
is the desired family. Indeed, $\diam H'<3 \delta/3=\delta$, $x_H\in H'\setminus \bigcup (\F \setminus \{H'\})$ and 
\[X\subseteq \bigcup \{\int H: H\in\H\}\cup\bigcup \{\int K:K\in\mathcal K\}\subseteq \bigcup \{\int H': H\in \H\}=\bigcup \{\int F: F\in\F\}.\]
\end{proof}

For the rest of the section, we will assume that $X$ is a fixed nondegenerate Peano continuum and $d$ is a fixed compatible convex metric on $X$ (every Peano continuum admits such a metric by \cite[Theorem 8]{Bing}). Under this setup, we denote for $\varepsilon>0$ 
\[
G_\varepsilon:= \bigcup_{\delta>0} \{f\in \C(X);\, \forall x_0,\dots,x_n \text{ a } \delta \text{-chain of }f \exists x\in X\forall i\leq n: d(x_i,f^i(x))<\varepsilon\}.
\]

\begin{lemma}\label{G_epsareopendense}
    For every $\varepsilon>0$ it holds that $\int (G_\varepsilon \cap \cl{\LC(X)}) $ is dense in $\cl{\LC(X)} $, where the interior is taken with respect to the subspace topology on $\cl{\LC(X)} $.
\end{lemma}

\begin{proof}
    Let $f\in \cl{\LC(X)}$ and $\varepsilon,\nu>0$, we will find $g\in \C(X)$, $\xi>0$ such that 
    \[
\emptyset\ne B(g,\xi)\cap\cl{\LC(X)} \subseteq G_\varepsilon\cap  \cl{\LC(X)}\cap B(f,\nu).
\]
    There exists $0<\delta<\min\{\varepsilon, \nu/4, \diam (X)\}$ such that if $A\subseteq X$, $\diam A<\delta$, then $\diam f(A)<\nu/4$. There is $h\in \LC(X)$ satisfying $d(f,h)<\nu/4$. Let $\F$ be a cover of $X$ as guaranteed by Lemma \ref{opencoverslebesgue} for $\delta$. Let $\lambda>0$ be the Lebesgue number of $\{\int F: F\in\F\}$. Consider arbitrary $F\in\F$, we have that $U_F':= F\setminus \bigcup (\F \setminus \{F\})$ is nonempty and open by Lemma \ref{opencoverslebesgue}. Thus $h$ is constant on some nonempty open set $U_F  \subseteq U_F'$. Fix a point $x_F \in U_F$. Repeat this process for every $F\in\F$. Fix $0<\xi\leq\min\{\lambda/2,\nu/4\}$ satisfying that $B(x_F,4\xi)\subseteq U_F$ for every $F\in \F$.

    For every $F\in \F$ fix a point $y_F\in U_F$ satisfying $d(x_F,y_F)= 3\xi$. Let $s_F:\cl{U_F} \to [0,3\xi]$ be given by $s_F(z):= d(z,X\setminus B(x_F,3\xi))$. Clearly $s_F$ is continuous and $s_F(\bd F)=\{0\}$ as $\bd F\subseteq X\setminus B(x_F,3\xi)$. Denote
    \[\F_F := \{F'\in\F;\,d(f(F),F')<\varepsilon/4\}.\]
    Let $t_F\in\LC([0,3\xi],\bigcup \F_F)$ be any continuous map satisfying:
\begin{itemize}
    \item $\{t_F(0)\}=h(U_F)$,
    \item for every $F'\in\F_F$ there is an interval $I_{F,F'}\subseteq [\xi,2\xi]$ such that $x_{F'},y_{F'} \in t_{F}(I_{F,F'})\subseteq \cl{B(x_{F'},3\xi)}$.
\end{itemize}
Note that this is possible since $d$ is convex and $h(U_F)\in \bigcup \F_F$ as $d(f,h)<\varepsilon/4$; to see that it is indeed possible to find $t_F$ in $\LC$, it may be convenient to use Lemma \ref{LCXY}. Finally, we define $g:X\to X$ by
     \begin{equation}
    g(x) :=
    \begin{cases}
      (t_F \circ s_F)(x) & \text{if } x\in \cl{U_F}, F\in \F, \\
      h(x)        & \text{otherwise.} 
    \end{cases}
  \end{equation}
     
    Note that $g$ is well-defined as if $x\in U_F$ for some $F\in \F$, then $F$ is unique, and continuous since $(t_F \circ s_F)(\bd U_F)=t_F(0)=h(U_F)$ for every $F\in\F$ by our choices. We will prove that $g$, $\xi$ have the required properties in a series of claims.

\begin{claim*}
    $B(g,\xi) \subseteq  B(f,\nu)$.
\end{claim*}

\begin{proofofclaim}
    Let $x\in X$. If $x\in U_F$ for some $F\in \F$, then by the construction $g(x)\in\bigcup \F_F$ and hence  the definition of $\F_F$ and $f(x)\in f(F)$ give
    \[d(f(x), g(x))\leq \diam f(F)+\nu/4+\mesh \F\leq \nu/4+\nu/4+\delta\leq 3\nu/4.\]
    Otherwise $g(x)=h(x)$ and recall that $d(f,h)<\nu/4<3\nu/4$.
    
    Thus $\xi<\nu/4$ concludes the proof.
\end{proofofclaim}

\begin{claim*}
    $g\in B(g,\xi)\cap\LC(X)$ and thus $\emptyset\ne B(g,\xi)\cap\cl{\LC(X)}.$
\end{claim*}

\begin{proofofclaim}
    Let $U\subseteq U_F$ for some $F\in\F$ be nonempty and open; by possibly making it smaller, we may assume that $U$ is connected since $X$ is Peano. If $|s_F(U)|=1$, then clearly $g$ is constant on $U$. Otherwise, $s_F(U)$ is a nondegenerate interval and thus $t_F$ is constant on an open set $V\subseteq s_F(U)$. Thus, $g$ is constant on an open nonempty set $s_F^{-1}(V)\cap U$. Hence, we conclude that $g\in \LC(X)$ since checking the other case is straightforward.
\end{proofofclaim}

    To prove that $B(g,\xi) \subseteq G_\varepsilon$, fix arbitrary $h\in B(g,\xi)$ and we will show that every $\lambda/2$-chain of $h$ is $\varepsilon$-shadowed. Let $x_0,\dots,x_n$ be a $\lambda/2$-chain of $h$, then it is a $\lambda$-chain of $g$ since $d(f,g)<\xi\leq\lambda/2$. Hence there exist $F_1,\dots,F_n\in \F$ such that $x_i,g(x_{i-1})\in F_i$ for every $i=1,\dots,n$ since $\lambda$ is the Lebesgue number of $\{\int F: F\in\F\}$. Let $F_0\in\F$ satisfy $x_0\in F_0$.
    
    \begin{claim*}
        It holds that $F_k\in \F_{F_{k-1}}$ for every $1\leq k\leq n$.
    \end{claim*}

    \begin{proofofclaim}
        We will distinguish cases $g(x_{k-1})\in U_{F_k}$ and $g(x_{k-1})\in F_k \setminus U_{F_k}$. Assume $g(x_{k-1})\in U_{F_k}$. Then by the choice of $U_{F_k}$ out of all sets in $\F$, only $F_k$ contains $g(x_{k-1})\in U_{F_k}$. Therefore $F_k\in \F_{F_{k-1}}$ as $g(F_{k-1})\subseteq \bigcup \F_{F_{k-1}}$ for $\F_{F_{k-1}}\subseteq \F$, by the definition of $\F_{F_{k-1}}$. Assume $g(x_{k-1})\in F_k \setminus U_{F_k}$. Then $g(x_{k-1})=f(x_{k-1}) \in F_k\cap f(F_{k-1})$, hence $F_k \in \F_{F_{k-1}}$ by the definition of $\F_{F_{k-1}}$. Thus both distinguished cases give $F_k \in \F_{F_{k-1}}$.
    \end{proofofclaim}
    
    We will prove by induction with respect to $k$ the following key claim:
    
    \begin{claim*}
        For every $0\leq k\leq n$ there exists $C\subseteq X$ such that $h^i(C)\subseteq U_{F_i}$ for every $0\leq i\leq k$ and $h^k(C)$ is a Peano continuum intersecting both $B(x_{F_k},\xi)$, $B(y_{F_k},\xi)$.
    \end{claim*}

    \begin{proofofclaim}
        If $k=0$, we can choose $C$ to be, for example, an arc with endpoints $x_{F_0},y_{F_0}$ isometric to the interval $[0,d(x_{F_0},y_{F_0})]$, there is some since $d$ is convex. Assume that the claim is true for $k-1$ and we will show that it is true for $k$. Thus, by the induction hypothesis, there exists $C'\subseteq X$ such that $h^i(C')\subseteq U_{F_i}$ for every $i=1,\dots,k-1$ and $h^{k-1}(C')$ is a Peano continuum intersecting both $B(x_{F_{k-1}},\xi)$, $B(y_{F_{k-1}},\xi)$.
        
        Let us examine $g(h^{k-1}(C'))$. Firstly, since $h^{k-1}(C')\subseteq U_{F_{k-1}}$, we have $g|_{h^{k-1}(C')}= (t_{F_{k-1}} \circ s_{F_{k-1}})|_{h^{k-1}(C')}$. Secondly, since $h^{k-1}(C')$ is a Peano continuum intersecting both $B(x_{F_{k-1}},\xi)$, $B(y_{F_{k-1}},\xi)$, necessarily $s_{F_{k-1}}(h^{k-1}(C'))$ is a connected set intersecting both $[0,\xi]$, $[2\xi,3\xi]$. Hence $[\xi,2\xi]\subseteq s_{F_{k-1}}(h^{k-1}(C'))$. Since $F_k\in \F_{F_{k-1}}$ by the previous claim, there is an interval $I_{F_{k-1},F_k}\subseteq [\xi,2\xi]\subseteq s_{F_{k-1}}(h^{k-1}(C'))$ such that $x_{F_k},y_{F_k} \in t_{F_{k-1}}(I_{F_{k-1},F_k})\subseteq \cl{B(x_{F_k},3\xi)}$ by the definition of $t_{F_{k-1}}$.
        
        By Lemma \ref{intermediatethrPeano} there is $K\subseteq h^{k-1}(L')$ a Peano continuum satisfying $s_{F_{k-1}}(K)=I_{F_{k-1},F_k}$. Hence $x_{F_k},y_{F_k} \in g(K)\subseteq \cl{B(x_{F_k},3\xi)}$ as $g(K)=t_{F_{k-1}}(s_{F_{k-1}}(K))=t_{F_{k-1}}(I_{F_{k-1},F-k})$. Thus $d(g,h)<\xi$ entails that $h(K)\subseteq B(x_{F_k},4\xi)\subseteq U_{F_k}$ is a Peano continuum intersecting both $B(x_{F_k},\xi)$, $B(y_{F_k},\xi)$. Let $C:= C'\cap h^{-k-1}(K)$. The set $C$ has all the properties desired since $h^k(C)=h(K)$.
    \end{proofofclaim}
    
    Note that, in particular, the claim implies the existence of a nonempty set $C$ such that $h^i(C)\subseteq U_{F_i}$ for every $0\leq i\leq n$. Finally, note that the orbit of any $x\in C$ indeed $\varepsilon$-shadows $x_0,\dots,x_n$ since $\mesh \F<\delta<\varepsilon$.
\end{proof}

\begin{theorem}\label{genericshadowing}
    Shadowing is generic in $\cl{\LC(X)}$ for every Peano continuum $X$.
\end{theorem}

\begin{proof}
    Let 
    \[
    G:= \bigcap\{\int(G_{1/n}\cap\cl{\LC(X)});\,n\in\N\},
    \]
    where the interiors are taken with respect to the subspace topology on $\cl{\LC(X)}$. Then $G$ is comeager in $\cl{\LC(X)}$ by Lemma \ref{G_epsareopendense}.
    It remains to check that $G\subseteq \SH(X)$. Let $f\in G$ and $\varepsilon>0$. We can find $n\in\N$ such that $1/n<\varepsilon$. Since $f\in G\subseteq G_{1/n}$, there is $\delta>0$ such that every $\delta$-chain of $f$ is $1/n$-shadowed and hence, in particular, $\varepsilon$-shadowed.
\end{proof}

\section{Acknowledgment}
The authors would like to thank Prof. Piotr Oprocha for his insightful remarks.

\bibliographystyle{alpha}
\bibliography{citace}
\end{document}